\documentclass[12pt]{article}
\usepackage{graphicx}  	
\usepackage{amsmath,amssymb} 
\usepackage{amsfonts}
\usepackage{amsthm}
\usepackage{bm}  		
\usepackage{pdfsync}  	
\usepackage{subfigure}
\usepackage{latexsym}
\usepackage[english]{babel}

\theoremstyle{plain}
\newtheorem{theorem}{Theorem}[section]
\newtheorem{proposition}[theorem]{Proposition}

\theoremstyle{definition}
\newtheorem{example}[theorem]{Example}
\newtheorem{definition}[theorem]{Definition}
\newtheorem{remark}[theorem]{Remark}
\newtheorem{algorithm}[theorem]{Algorithm}

\newcommand{\R}{{\mathbb R}}
\newcommand{\N}{{\mathbb N}}

\newcommand{\Z}{{\mathcal Z}}
\newcommand{\X}{{\mathbb X}}

\newcommand{\Xtilde}{{\widetilde{\X}}}
\newcommand{\Xeps}{{\X^{\varepsilon}}}
\newcommand{\I}{{\mathcal I}}
\newcommand{\QB}{{\mathcal O}}
\newcommand{\Mat}{{\rm Mat}}
\newcommand{\rank}{{\rm rank}}

\newcommand{\ebold}{{\bf e}}

\newcommand{\ie}{{\it i.e.}}
\def\cocoa{{\hbox{\rm C\kern-.13em o\kern-.07em C\kern-.13em o\kern-.15em A}}}

\def\Mat{\mathop{\rm Mat}\nolimits}

\def\rank{\mathop{\rm rank}\nolimits}

\def\y1{\mathbf{y}^{(1)}}

\numberwithin{equation}{section}

\title{Simple Varieties for Empirical Points}
\author{ Claudia Fassino$^*$  Maria-Laura Torrente\thanks{Dip. di Matematica, Universit\`a di Genova, via Dodecaneso 35, 16146 Genova, Italy  (\tt{fassino@dima.unige.it}, \tt{torrente@dima.unige.it})}}
\date{}

\begin{document}
\maketitle

\begin{abstract}
\begin{abstract}
We present  a symbolic-numeric approach for the analysis of a given set of noisy data, represented as a finite set~$\X$ of limited precision points.  Starting from~$\X$ and a permitted tolerance~$\varepsilon$ on its coordinates, our method automatically determines a low degree monic polynomial whose associated variety passes close to each point of $\X$ by less than the given tolerance~$\varepsilon$. 
\end{abstract}
\end{abstract}


\section{Introduction}\label{intro}
It is a well-known matter that the analysis and modeling of real-world phenomena 
often relies upon measurements and observations which give rise to sets of 
data represented as affine sets of points.   
An algebraic way that exploits the collected data to construct 
a mathematical model of the observed phenomenon consists in computing the
vanishing ideal of the affine set of points, that is the ideal comprising all 
polynomials which vanish at the given points. The vanishing ideal is classically 
determined using the Buchberger-M\"oller (BM) Algorithm \cite{BM82}, a low complexity method 
which returns a Gr\"obner basis of it.
Nevertheless, the fact that data and relations from the real world are essentially characterized 
by a limited accuracy makes this algebraic approach unfeasible in practice. In fact, even 
small perturbations of the points end up with very different vanishing ideals, 
which may have completely different bases, since the question whether or not 
a polynomial vanishes at a given set of points is obviously very sensitive to
perturbations. For instance, this is a well-known phenomenon when 
using Gr\"obner basis theory (\cite{KR05}, \cite{M99}) which turns out to be unsuitable 
as a numerical tool.
Thus, when dealing with real-world phenomena, the sole use of the classical 
algebraic techniques turns out to be insufficient, and a combination with 
numerical tools is required.

Motivated by the need of describing real-world processes,  
the development of a numerical branch within a discrete mathematical discipline already took place in 
different areas of mathematics, such as Linear Algebra, Differential Equations, Optimization, etc. 
In the nineties, classical nonlinear algebra began to follow this trend, also thanks to the work of H.~J. Stetter, 
whose book~\cite{St} is considered as a stepping stone at the interface of symbolic and numerical computation.
This new emerging discipline has been given different names (see also~\cite{RA}); in~\cite{St06}
its mathematical content has been clearly described as including the ``areas of commutative algebra 
not over arbitrary fields or rings but over the analytically structured fields of the real or complex numbers",
in which the real or complex metric is taken into consideration. 
Obviously, in the attempt of generalizing the notions of classical Commutative Algebra to the 
real or complex coefficients, the greatest attention must be paid regarding the new concepts and
terminology. In particular, in our case, the classical notion of vanishing ideal turns out to be
inadequate for limited precision points, and so it needs to be replaced by an 
appropriate concept also suitable for numerical computations.

In the literature a notion of vanishing ideal of a set of limited 
precision points already exists; in~\cite{AFT08}, \cite{Fa10}, \cite{HKPP09}, \cite{KP11}, \cite{S07}, \cite{T}
it is defined in different ways as an ideal containing polynomials whose evaluations 
at the original points assume small values. 
Note that, according to this definition, the vanishing ideal of a set of 
limited precision points is not necessarily zero-dimensional; further, 
maximum relevance is given to the evaluations of the polynomials 
at the points, while the geometrical distance of the zero-locus 
of the polynomials from the original points is not taken into consideration.
In this paper we follow a different approach: we give more relevance to the distance 
points-variety and introduce the idea that the vanishing ideal of a set of
limited precision points is generated by polynomials of low degree whose affine variety 
lies close to the original points. Among such polynomials, a crucial role is played by those 
of minimal degree, which provide an immediate interpretation of the phenomenon 
encoded by~$\X$ and a simple geometrical 
representation of the data points, as shown in the following example.

\begin{example}\label{mainExample}
Let $\X \subset \R^2$ be a set of points created by perturbing by less than~$0.1$ the coordinates of~$10$ points lying on the affine variety $g=0$ where ${g=y^2-x-2y+2}$:
\begin{equation*}
\begin{array}{rcl}
\X &=& \{(0.95, 1), (1.3, 0.5), (2.05, 1.98), (2.08, 0), (3.18, -0.48),\\
&& (5.05, 2.95),(5.05, -0.95), (7.2, -1.45), (9.98, 4),(10.05, -2)\}
\end{array}
\end{equation*}
Using standard techniques of Computational Algebra we obtain that the 
minimal degree of the  polynomials vanishing at~$\X$ is $4$;
this degree is too high to point out that the set $\X$ lies close to a parabola,
while $g$ and its zero-locus $\Z(g)$, almost crossing each point of $\X$, 
provide a good numerical representation of the given data set.
\end{example}

This paper, which is the first part of a wider investigation, reports on the problem of computing a polynomial $f$ of low degree whose associated affine variety $\Z(f)$, which is defined as the set 
of points at which the polynomial~$f$ vanishes, almost crosses each element of the data set~$\X$. 
The final aim of our future work is to detect a set of polynomials whose minimal 
degree is strictly bounded by the minimal degree of the elements of the vanishing 
ideal of~$\X$ and such that their zero-set is made up of points each differing from 
the corresponding point of~$\X$ by less than the given tolerance.
Put differently, we think to the {\it numerical} vanishing ideal as a truly zero-dimensional
ideal whose {\it exact} zeros are small perturbations of the original limited precision
points, and so numerically indistinguishable from them. 
An evident merit of this point of view  is that 
the gain in simplicity due to the lower degree of the polynomials might offset the noise in the data 
and help to discover simpler laws that rule the phenomenon we aim to describe.

In particular, in this paper we address the following problem: 
given a finite set $\X$ of points and a permitted positive tolerance~$\varepsilon$ on the coordinates of each point, 
we look for a polynomial~$f$ whose degree is strictly bounded by the minimal degree of the elements of the vanishing 
ideal~$\I(\X)$ of~$\X$ and whose affine variety~$\Z(f)$ lies close to the points of~$\X$ by less than~$\varepsilon$.
We present a new algorithm (the Low-degree Polynomial Algorithm - LPA) that, starting from 
$\X$ and $\varepsilon$, uses some peculiarities of the SOI algorithm \cite{AFT08} and NBM algorithm~\cite{Fa10}, both theoretically based on the BM algorithm, to determine a polynomial $f$ of suitable degree.
In most cases $f$ also satisfies some simple conditions,  derived from Kantorovich theorem~\cite{WW},  under which the variety~$\Z(f)$ is proved to lie close to the points of~$\X$ by less than $\varepsilon$. In these favourable cases $f$ is a solution of our problem.

More precisely, partly parallelizing the procedures presented in~\cite{AFT08} and~\cite{Fa10}, we reformulate the addressed  problem as the problem of solving  an ordered finite sequence $\{F_i = 0\}_i$ of nonlinear systems subject to constraints.
The first  nonlinear system that admits an approximated solution satisfying the constraints allows the direct computation of the polynomial $f$.
In the LPA we solve each nonlinear system $F_i=0$ using an iterative method which is a modified version of  the classical Normal Flow algorithm \cite{WW}. Each iteration requires the solution of a linear system whose coefficient matrix is the evaluation of the Jacobian $J_{F_i}$ at a suitable point. Since an analytic expression of $F_i$ as well as of $J_{F_i}$ is computationally very hard to obtain, the evaluation of $J_{F_i}$ is performed using a  new method that does not require the explicit knowledge of $F_i$. Due to the technicalities of this method, its details are separately  illustrated 
inside the Appendix. 
Moreover, since the evaluation of the Jacobian $J_{F_i}$  often turns out to be an ill-conditioned matrix, the corresponding linear system is solved using a method based on the Rank Revealing decomposition \cite{HP92}.

This paper is organized as follows. 
In order to be self-consistent and to fix notations, in Section~\ref{Definitions} we introduce some basic  concepts of Computational Commutative Algebra and Numerical Analysis.
In Section~\ref{underdetNonlinearSystem}  we introduce the modified version of the Normal Flow Algorithm to find an approximated solution of an underdetermined nonlinear system subject to constraints, paying particular attention to the ill-conditioned case.
Section~\ref{mainResult} presents the main results of the paper: the Low-degree Polynomial Algorithm, which performs the main steps of the method described above, and 
Theorem~\ref{theoremBasedOnKantorovich}, which proves that, under certain hypotheses, the output of the LPA is a solution of our theoretical problem.
In Section~\ref{examples} we give numerical examples illustrating the functioning of the LPA.
Finally,~\ref{appendix} contains the basic result for the evaluation of the Jacobian matrix associated with our nonlinear function, without passing through the explicit expression of it.

\goodbreak
\section{Preliminary definitions and results}\label{Definitions}
We recall some basic notation about matrices. 
Let $m,n \ge 1$ be integers and $\Mat_{m \times n}(\R)$ be the set of $m \times n$ matrices with entries in~$\R$; if $m=n$ we simply write $\Mat_n(\R)$.
Let $A\in \Mat_{m \times n}(\R)$; we denote by~$\|A\|_p$ the $p$-norm of $A$, $p=1,2,\infty$.

We divide the rest of this section into three  subsections: the first one contains some preliminary definitions about the polynomial ring $\R[x_1,\ldots,x_n]$ and the notion of empirical points; the second one reports some basic results on the numerical rank of a matrix; the third one describes a classical numerical approach to solve underdetermined nonlinear systems.

\subsection{Preliminary definitions}
We recall some concepts related to the polynomial ring $P=\R[x_1,\ldots,x_n]$ (\cite{KR00},~\cite{KR05}).

\begin{definition}\label{eval}
Let $\X= \{p_1,\dots,p_s\}$ be a non-empty finite set of points of $\R^n$, $f$ be a polynomial and $G=\{g_1,\dots,g_k\}$ be a non-empty finite set of polynomials.
\begin{enumerate}
\item $\I(\X)=\{f \in P : f(p_i)=0 \; \forall p_i \in \X\}$ is the {\bf vanishing ideal} of~$\X$.
\item $\Z(f)=\{p \in \R^n : f(p)=0\}$ is the {\bf affine variety} associated with~$f$.
\item The $\R$-linear map ${\rm{eval}}_{\X} : P \rightarrow \R^s$
defined by ${\rm{eval}}_{\X}(f) = (f(p_1),\dots,f(p_s))$
is called the {\bf evaluation map} associated with $\X$.
For brevity, we write $f(\X)$ to mean ${\rm eval}_{\X}(f)$.
\item The {\bf evaluation matrix} of $G$ associated with $\X$, denoted by
$M_{G}(\mathbb X) \in \Mat_{s\times k}(\R)$, is defined as having entry $(i,j)$
equal to $g_j(p_i)$, \ie~whose columns are the images of the polynomials
$g_j$ under the evaluation map.
\end{enumerate}
\end{definition}
Let  $t=x_1^{\beta_1}\dots x_n^{\beta_n}$, $\beta_i \in \N$, be a power product  and $\QB$ be a set of power products;  we denote by $\partial_k t =\beta_k x_1^{\beta_1}\dots x_k^{\beta_k-1}\dots x_n^{\beta_n}$ the {\bf $k$-th formal partial derivative} of $t$ and by $\partial_k \QB = \{\partial_k t \::\: t \in \QB\}$.

\smallskip
We formalize the idea of perturbed point by introducing a simplified version of the notions of empirical point, admissible perturbation and almost vanishing polynomials (see~\cite{AFT07}, \cite{AFT08}, \cite{Fa10}, \cite{St}).
To this aim we recall that given a real value $\eta \ll 1$ and a real function~$w(x)$ we say that $w(x)=O(\eta^k)$, $k \in \N$, if $|w(x)|/\eta^k$ is bounded near the origin.

\goodbreak
\begin{definition}\label{empirPts}
Let $p,q \in \R^n$ be points and $\varepsilon >0$.
\begin{enumerate}
\item The pair $(p,\varepsilon)=p^\varepsilon$ is called an {\bf empirical point}:~$p$ is the {\bf specified value} and~$\varepsilon$ is the {\bf tolerance} of~$p^\varepsilon$.
\item The {\bf $\varepsilon$-neighborhood} of~$p^\varepsilon$ is defined as ${N(p^\varepsilon) = \{\widetilde{p} \in \R^n :\|\widetilde{p}-p\|_\infty \le \varepsilon\}}$
and contains all the {\bf admissible perturbations} of~$p^\varepsilon$.
\item  Two empirical points $p^\varepsilon$, $q^\varepsilon$ are said to be {\bf distinct} if $ N(p^\varepsilon)  \cap N(q^\varepsilon)  = \emptyset $.
\end{enumerate}
\end{definition}

\begin{definition}\label{appfun}
Let $\Xeps=\{p_1^\varepsilon,\ldots, p_s^\varepsilon\}$ be a set of empirical points with uniform tolerance~$\varepsilon$ and with $\X=\{p_1,\ldots,p_s\} \subset \R^n$
\begin{enumerate}
\item   A set of points $\Xtilde=\{\widetilde{p_1},\ldots,\widetilde{p_s}\} \subset \R^n$ is called an {\bf admissible perturbation} of~$\Xeps$ if each~$\widetilde{p_i} \in N(p_i^\varepsilon)$.
\item  A polynomial $g$  is called {\bf almost vanishing} at $\Xeps$ if ${\| g(\X)\|_2 / \|g\| = O(\varepsilon)}$, where $\|g\|$ is the $2$-norm of its coefficient vector.
\end{enumerate}
\end{definition}

\subsection{The numerical rank}\label{num_r}
In this section we recall the definition of the numerical rank of a matrix $A \in \Mat_{s \times t}(\R)$, and some basic results related to this topic (\cite{GVL96}, \cite{HP92}). 
In the case $s \ge t$ such results will allow us to detect  a partitioning of the columns 
of $A$ (if $s < t$ the partitioning involves the rows of $A$) into two  submatrices $A_1$ and $A_2$ ($A_2$ eventually the empty matrix) such that $A_1$ is  well-conditioned and  $A_2$ consists of columns  (or rows) which are ``almost" depending on the columns (or rows) of $A_1$.

Let $A \in \Mat_{s \times t}(\R)$; we denote by $\rank(A)$ the rank of~$A$, by $\sigma_i(A)$ the $i$-th singular value of~$A$, by $A^\dagger$ the pseudoinverse of~$A$ and by $K_2(A)=\|A\|_2 \|A^{-1}\|_2= \sigma_1(A)/\sigma_{\min\{s,t\}}(A)$ the $2$-norm condition number of $A$. We denote by $0$ the vector or the matrix whose elements are equal to $0$ and whose dimension is deducible from the context.

Given $k>1$ and a small threshold $\delta$ we say that $A$ has numerical $(\delta,k)$-rank $r$ if there exists a well-determined gap between $\sigma_r(A)$ and $\sigma_{r+1}(A)$ and if both~$\delta$ and~$k \delta$ lie in this gap. 
This concept can be formalized as follows.
\begin{definition}\label{defNumericalRank}
Let $r \le \min\{ s,t\}$ and $A \in \Mat_{s \times t}(\R)$; let $\delta >0$ and $k > 1$. The {\bf numerical $(\delta,k)$-rank} of $A$, denoted by $\rank_{\delta,k}(A)$, is equal to $r$ iff
$$
\sigma_1 \ge \ldots \ge \sigma_r(A) > k \delta > \delta >  \sigma_{r+1}(A) \ge \ldots \ge \sigma_{\min\{s,t \}}(A)
$$
\end{definition}
The singular values of $A$ might not clearly split into small and large subsets making the determination of the numerical rank somewhat arbitrary. This leads to more complicated methods for estimating $\rank_{\delta,k}(A)$ which involve the LS problem\cite{GVL96}. In the following of the paper we assume that either the matrix $A$ has full numerical rank or there exists a well-determined gap between $\sigma_r(A)$ and $\sigma_{r+1}(A)$.

\begin{remark}\label{cond}
If $A \in \Mat_{s \times t}(\R)$ is a  full numerical $(\delta,k)$-rank matrix then it  is well-conditioned, that is its condition number cannot be too elevated, as
$K_2(A) = \frac{\sigma_1(A)}{\sigma_{\min\{s,t\}}(A)} <  \frac{\sigma_1(A)}{k\delta}$. 
On the contrary, if $\rank_{\delta,k}(A)=r<{\min\{s,t \}}$ then the matrix is ill-conditioned, that is its condition number is elevated, since
$K_2(A)= \frac{\sigma_1(A)}{\sigma_{\min\{s,t \}}(A)} >  \frac{\sigma_1(A)}{\delta}$. 
\end{remark}

The numerical $(\delta,k)$-rank of $A$ points out that there are no matrices with exact rank less than $\rank_{\delta,k}(A)$ which differ from $A$ by less than $\delta$, as  the following proposition shows.

\begin{proposition}\label{consequenceNumRank}
Let $A \in \Mat_{s \times t}(\R)$ with $\rank_{\delta,k}(A)=r$.\\
Then the set $A_\delta=\{M \in \Mat_{s \times t}(\R) \: : \: \|M-A\|_2 < \delta\}$ contains at least one matrix of rank~$r$ but no matrix of rank strictly less than~$r$.
\end{proposition}

\begin{proof}
From $\rank_{\delta,k}(A)=r$ it follows that $\sigma_r(A) > \delta > \sigma_{r+1}(A)$; further, from the Eckart-Young theorem~\cite{GVL96} for each $j=1 \ldots \min\{s,t\}$ we have $\sigma_j(A) = \min \{ \| A-B \|_2 : \rank(B) = j-1\}$.
Let $C \in \Mat_{s \times t}(\R)$ be s.t.~$\rank(C)=r$ and 
$\|A-C\|_2=\min \{ \| A-B \|_2 :  \rank(B) = r\}$; since $\sigma_{r+1}(A) < \delta$ then $C \in A_\delta$.
Further, since for each $j \le r$ we have $\sigma_j(A) > \delta$; it follows that $A_\delta$ contains no matrix of rank strictly less than~$r$.
\end{proof}

The numerical rank of $A$ is preserved under small perturbations, since the following result about the sensitivity of the singular values holds.
\begin{proposition}\label{lemmaSigma}
Let $A, E \in \Mat_{s\times t}(\R)$; then for each $1\le j \le \min\{s,t\}$ we have
$|\sigma_j(A+E)- \sigma_j(A) | \le \sigma_1(E)= \|E\|_2$
\end{proposition}
\begin{proof} See~\cite{GVL96}, Corollary 8.3.2
\end{proof}

We end this section with  Theorem~\ref{theoremRRQR}~\cite{HP92} which 
has important consequences when applied to the case $r=\rank_{\delta,k}(A)$.         

\begin{theorem}\label{theoremRRQR}
Let $r \le t \le s$ and $A \in \Mat_{s \times t}(\R)$.
There exists a permutation matrix $\Pi \in \Mat_t(\R)$ such that
\begin{equation}\label{RRQR}
A \Pi =\left (Q_1 \;|\: Q_2\right ) \left ( \begin{array} {cc} R_{11} & R_{12} \\ 0 & R_{22} \end{array} \right )
\end{equation}
where $Q=\left (Q_1 \;|\: Q_2\right )$ is orthonormal, $Q_1\in \Mat_{s \times r}(\R)$, $Q_2\in \Mat_{s\times t-r}(\R)$, $R_{11} \in \Mat_r(\R)$ and $R_{22} \in \Mat_{t-r}(\R)$ are upper triangular, and
\begin{equation}\label{RRQRcondition}
\sigma_{\min}(R_{11}) \ge \frac{\sigma_r(A)}{q(t,r)} \quad \textrm{and } \quad
 \|R_{22}\|_2 \le q(t,r) \sigma_{r+1}(A)
\end{equation}
where $q(t,r)=\sqrt{r(t-r) + \min(r,t-r)}$.
\end{theorem}
\begin{proof} See~\cite{HP92}, Theorem 2.2.
\end{proof}

 Theorem~\ref{theoremRRQR}, when applied to the case  $r=\rank_{\delta,k}(A)$, states the existence of a set of $r$ 
 strongly independent columns of $A$ or, equivalently, of a well-conditioned submatrix $A_1\in \Mat_{s \times r}(\R)$ of $A$. In particular, it proves that there exist a permutation matrix $\Pi$ and a partitioning of $A\Pi=\left( A_1 \;| \; A_2 \right )$ where $ A_1= Q_1R_{11}  \in \Mat_{s \times r}(\R)$ and   $A_2 =  Q_1R_{12} +Q_2 R_{22} \in \Mat_{s \times t-r}(\R) $, which has the following  interesting numerical  properties.
If the  gap between $\sigma_r(A)$ and  $\sigma_{r+1}(A)$ is large enough, that is if 
$\sigma_r(A) \gg \sigma_{r+1}(A)$, then $A_1$ is a submatrix of $A$ consisting  of the  maximum number of strongly independent columns w.r.t.~the threshold $\delta$. 
In fact,  from $\sigma_r(A_1)=\sigma_r(R_{11}) > \frac{\sigma_r(A)}{q(t,r)} >0$, we have that
$A_1$ has full rank,  and from
$Q_1=A_1 R^{-1}_{11}$, $\|R_{22}\|_2 \le  q(t,r) \sigma_{r+1}(A)$ and 
 $A_2 = Q_1R_{12} +Q_2 R_{22}$
we have that each column of the matrix $A_2$  can be expressed as the sum of a linear combination of the columns of $A_1$ and a vector whose $2$-norm is less than $q(t,r)\sigma_{r+1}(A) $.
Furthermore, though $A$ is ill-conditioned (see Remark~\ref{cond}), $A_1$ is well-conditioned, since  
${K_2(A_1) = \frac{\sigma_1(A_1)}{\sigma_r(A_1)} = \frac{\sigma_1(A_1)}{\sigma_r(R_{11})} \le   \frac{\sigma_1(A)}{\sigma_{r}(A)} q(t,r)}$

\subsection{Underdetermined nonlinear systems}\label{Under_sys}
Let $D \subseteq \R^n$, $F: D \rightarrow \R^m$, with $m<n$, be a differentiable nonlinear function, $J_F(x)$ be the Jacobian matrix of~$F$ at~$x$ and let $F=0$ be the underdetermined nonlinear system to solve.
In this subsection we  recall the Normal Flow Algorithm, a classical iterative method for approximating a solution of $F=0$ (see~\cite{WW} and the references given there).
 
\begin{algorithm}{\bf (The Normal Flow Algorithm - NFA)}\label{normalFlowAlg}\\
Let $D \subseteq \R^n$, $F: D \rightarrow \R^m$ be a differentiable nonlinear function, $\omega \ll 1$ be a fixed threshold, and $\bar x \in D$ be the initial point. Consider the following sequence of steps.
\begin{enumerate}
\item[{\bf NF1}] Let $h=(1,\ldots,1)^t$.
\item[{\bf NF2}] While $\|h\|_2 > \omega$ \\
\phantom{aa} $\bullet$ Compute the minimal $2$-norm solution of  $J_F(\bar x) h = -F(\bar x)$.\\
\phantom{aa} $\bullet$ Let $\bar x=\bar x+h$.\\
Return $\bar x$ and stop.
\end{enumerate}
\end{algorithm}

At each iteration the minimal $2$-norm solution of the underdetermined linear system $J_F(\bar x) h = -F(\bar x)$ can be 
either computed using $J_F^\dagger(\bar x)$ or, more efficiently, performing a $QR$ decomposition 
of $J_F(\bar x)$ (see \cite{DH93}).

\smallskip
The next classical result (see~\cite{WW}) is a local convergence theorem for the NFA that also provides sufficient conditions for the existence of a local solution of $F=0$.

\begin{theorem}\label{KantorovichTheorem}{\bf (Kantorovich theorem)}\\
Let $D \subseteq \R^n$ be an open set, $F: D \rightarrow \R^m$ be a differentiable nonlinear function and~$J_F(x)$ be of full rank~$m$ in an open convex set $\Omega \subseteq D$. Let $x_0 \in \Omega$, $\eta > 0$ and $\Omega_\eta = \{x \in \Omega \::\: \|y -x\|_2 < \eta \Rightarrow y \in \Omega\}$; suppose that
\begin{enumerate}
\item[$\bullet$] $\exists\;  \gamma \ge 0$, $p \in (0,1]$ such that $\|J_F(y) - J_F(x)\|_2 \le \gamma \|y-x\|_2^p$, $\forall x,y \in \Omega$
\item[$\bullet$] $ \exists  \; \mu > 0$ such that $\|J_F^\dagger(x)\|_2 \le \mu$, $\forall  x \in \Omega$
\end{enumerate}
Then there exists $\beta >0$ satisfying
$\tau = \frac{\gamma \mu^{1+p}\beta^p}{1+p} < 1$ 
and $\frac{\mu \beta}{1 - \tau} < \eta$
such that if $x_0 \in \Omega_\eta$ and $\|F(x_0)\|_2 < \beta$ then the iterates $\{x_k\}_{k=0,1,\ldots}$ determined by the NFA are well defined and converge to a point $x^* \in \Omega$ such that $F(x^*)=0$.
\end{theorem}
\begin{proof} See~\cite{WW}, Theorem~2.1.
\end{proof}

\section{Underdetermined and ill-conditioned nonlinear systems}\label{underdetNonlinearSystem}

As reported in the Introduction, we reformulate  the problem addressed in this paper as the problem of solving an ordered finite sequence of  underdetermined nonlinear systems subject to constraints which, in general,  turn out to be ill-conditioned;  in this section we present a method for solving such kind of  nonlinear systems.  

Let $D \subseteq \R^n$ be a set containing the origin, let $F: D \rightarrow \R^m$ be a nonlinear function of class~$\mathcal C^2(D)$,   and let $Q_\varepsilon=\{x \in \R^n \: : \: \|x\|_\infty \le  \varepsilon\}$; our  aim is to compute a solution of the underdetermined nonlinear system $F=0$ subject to $x \in Q_\varepsilon\cap D$.
As reported in Section~\ref{Under_sys} a classical method for finding an approximated solution of $F=0$ is given by 
the NFA, which returns successively better approximations of a zero of the system $F=0$ by computing at each step the minimal $2$-norm solution of $J_F(\bar x) h = - F(\bar x)$, where $\bar x$ is a suitable point of $\R^n$. Obviously, when  the nonlinear system $F=0$ is ill-conditioned, that is when the 
condition number of the Jacobian matrix $J_F$ of $F$ evaluated in the constrained region is too elevated, 
the NFA could end up with unreliable solutions. In order to overcome  the numerical instabilities in the computations and avoid the unreliable solutions which can occur in the ill-conditioned case, we present an alternative method that firstly replaces $F$ with a new suitable function $\widehat F$ slightly differing from $F$ in the constrained region, that is such that $\|\widehat F(x) - F(x) \|_2 $ is small on $Q_\varepsilon\cap D$, and well-conditioned at the origin (and also in a neighbourough of it). If this latter condition is not met, that is if the Jacobian matrix $J_{\widehat F}(0)$ is ill-conditioned, we prove under a simple additional  hypothesis that the original system $F=0$  subject to constraints has no solution (see Theorem~\ref{noExistenceWellCondCase}).
Successively our method applies a  a modified version of the NFA to the system $\widehat F =0$; the iterations of this modified version are stopped either when the computed deplacement is small enough, as in the classical version, or when the current iteration goes out of the region where $J_{\widehat F}$ is well-conditioned.  
At the end this new algorithm returns either the current iterate $\bar x$ if the given constraints are satisfied, or  a warning message.
This  method, that we present as the Root Finding Algorithm (see Algorithm~\ref{algRF}), turns out to be  {\it backward stable}  \cite{Datta10} by construction, since it tries to solve, instead of the original constrained system, the nearby problem $\widehat F=0$ with the same constraints.

In order to define the new function ${\widehat F: D \rightarrow \R^m}$, we  use the numerical rank of the matrix $(J_F(0) \:|\: F(0)) $ as follows.  Given a threshold $\delta \ge \varepsilon$ of the same order of magnitude as~$\varepsilon$ and an integer $k>1$, 
we detect the numerical $(\delta,k)$-rank of the matrix $(J_F(0)\:|\: F(0) )^t$ by means of the SVD decomposition.  Let $\Pi$ be the permutation matrix whose existence is stated in Theorem~\ref{theoremRRQR} with $r=\rank_{\delta,k}(J_F(0)\:|\: F(0) )^t$, which satisfies

\begin{equation}\label{RRQRofJF}
\left ( \begin{array}{c}J_F(0)^t \\ F(0)^t \end{array} \right) \Pi =
\left (Q_1 \;|\: Q_2\right )\left ( \begin{array} {cc} R_{11} &  R_{12} \\ 0 &  R_{22} \end{array} \right )
\end{equation}
where $Q=\left (Q_1 \;|\: Q_2\right )$ is orthonormal, $Q_1\in \Mat_{n+1 \times r}(\R)$, $Q_2\in \Mat_{n+1 \times m-r}(\R)$, 
$R_{11} \in \Mat_r(\R)$ and $R_{22} \in \Mat_{m-r}(\R)$ are upper triangular.
Further we have $\sigma_{\min}(R_{11}) \ge (k/q(m,r)) \delta$ and $\|R_{22}\|_2 \le q(m,r)\delta$, where $q(m,r)=\sqrt{r(m-r)+\min(r,m-r)}$.
Let$F_1: D  \rightarrow \R^r$
and $F_2: D \rightarrow \R^{m-r}$
be the functions made up of the first~$r$ and the last~$m-r$ components of~$\Pi^tF$.
The function $\widehat F:D \rightarrow \R^m$, defined  as
\begin{equation}\label{HatF}
\widehat F(x) \equiv\left ( \begin{array}{c}
\widehat F_1(x) \\
\widehat F_2(x)
\end{array}\right ) =
\left ( \begin{array}{c}
 F_1(x) \\
G F_1(x) \end{array}\right) \;\;\textrm{with }\;\; G=R_{12}^t  R_{11}^{-t}
\end{equation}
is ``close" to the function $F$, as shown in the following theorem.

\begin{theorem}\label{Differenza_F}
Let $\Delta \subseteq Q_\varepsilon \cap  D$ be a closed convex set containing the origin and let $\widehat F$ be the function defined by~(\ref{HatF}); for each $x \in \Delta$ we have
$$
\| \widehat F(x) -\Pi^t F(x)\|_2  \le q(m,r)\delta + O(\delta^2)
$$
Moreover, if there exists $x^*\in \Delta$ such that $F(x^*)=0$, then
$$
\| \widehat F(x) -\Pi^t F(x)\|_2 = O(\delta^2)
$$
\end{theorem}

\begin{proof}
W.l.o.g.~in~(\ref{RRQRofJF}) we may assume $\Pi$ equals to the identity matrix; using~(\ref{RRQRofJF}) and an obvious partitioning of $Q=(Q_1 \;|\; Q_2)$ the matrix $(J_F(0) \:|\: F(0)  )$ can be rewritten as
\begin{equation}\label{RRDec}
\left (J_F(0) \; |\; F(0) \right )=\left ( \begin{array}{cc}
J_{F_1}(0) &F_1(0) \\
J_{F_2}(0)& F_2(0)
\end{array}\right )=\left ( \begin{array}{l}
R^t_{11} Q^t_1 \\
R^t_{12} Q^t_1 + R^t_{22} Q^t_2 \\
\end{array}\right )
\end{equation}
We denote by $d(x)= (\widehat F_2-F_2)(x)=(G F_1 -F_2)(x)$ and so  from~(\ref{HatF}) we have that
$\|\widehat F(x) -F(x) \|_2 =\|\widehat F_2(x) -F_2(x) \|_2=\| d(x)\|_2$.
Further from~(\ref{RRDec})
\begin{equation*}
\begin{array}{lll}
(J_d(0) \;|\; d(0) ) &=& ( GJ_{F_1}(0) -J_{F_2}(0)\;|\; GF_1(0) -F_2(0)) \\
&=& G( J_{F_1}(0) \;|\; F_1(0)) -( J_{F_2}(0)\;|\; F_2(0))  \\
&=& G R_{11}^t Q_1^t - R_{12}^tQ_1^t -R_{22}^t Q_2^t = -R_{22}^t Q_2^t
\end{array}
\end{equation*}
therefore $\|J_d(0)\|_2 \le q(m,r) \delta$  and $\|d(0)) \|_2 \le q(m,r)\delta$.

We consider the Taylor expansion of $d(x)$ at the origin; since $d \in \mathcal C^2(\Delta)$ and  $\Delta \subseteq Q_\varepsilon \cap D$ is a closed and convex set, we have $d(x) = d(0)+J_d(0) x +O(\varepsilon^2)$, for each $x \in\Delta$. Since $\|x\|_2 \le\sqrt{n}\varepsilon$ and $\delta \ge \varepsilon$ we have
\begin{equation*}
\begin{array}{rcl}
\|d(x)\|_2 &\le&  \|d(0)\|_2+\|J_d(0)\|_2 \|x\|_2 + O(\varepsilon^2) \\ 
&\le& \|d(0)\|_2+q(m,r) \sqrt{n}\delta \varepsilon  + O(\varepsilon^2) \\
&\le & \|d(0)\|_2+q(m,r) \sqrt{n}\delta^2  + O(\delta^2) = \|d(0)\|_2+ O(\delta^2)
\end{array}
\end{equation*}
From $\|d(0)) \|_2 \le q(m,r)\delta$ the first inequality follows.

Now, let $x^*\in \Delta$ such that $F(x^*)=0$; evaluating the previous Taylor expansion of~$d(x)$ at~$x^*$ we get $0=d(0)+J_d(0)x^*+O(\varepsilon^2)$, and so $d(0) = -J_d(0)x^*+O(\varepsilon^2)$.
Since $\|x^*\|_2 \le\sqrt{n}\varepsilon$, it follows that
$$
\|d(0)\|_2 \le q(m,r) \sqrt{n}\delta \varepsilon +O(\varepsilon^2) \le  q(m,r) \sqrt{n}\delta^2 +O(\delta^2) = O(\delta^2)
$$
and this concludes the proof.
\end{proof}

From Theorem~\ref{Differenza_F}, the partition given in (\ref{RRQRofJF}) can be used in the Root Finding Algorithm to derive the new function $\widehat F$. 

Since, by construction, the last $m-r$ equations of the nonlinear system $\widehat F=0$ linearly depend on the first  $r$ equations, in order to approximate the solution of $\widehat F =0$ we only have to  consider the nonlinear system $F_1=0$.

From formula~(\ref{RRQRcondition}), since $\left( J_{F_1}(0) \;|\; F_1(0) \right ) = R^t_{11} Q_1^t$  it follows that 
$$\sigma_r\left( (J_{F_1}(0) \;|\; F_1(0) )\right ) = \sigma_{\min}  (R_{11}) > \frac{\sigma_r\left( (J_{F}(0) \;|\; F(0)) \right ) }{q(m,r)}$$
and so $\left( J_{F_1}(0) \;|\; F_1(0) \right )$ is well-conditioned if the  $r$-th singular value of the matrix $\left( J_{F}(0) \;|\; F(0) \right )$ is large enough. 
Nevertheless this fact  does not guarantee that the matrix $ J_{F_1}(0) $ is well conditioned too. For  this reason, 
the Root Finding Algorithm, after the construction  the  function $F_1$, checks whether the numerical rank of $ J_{F_1}(0)$ is equal to $r$.    If $\rank_{\delta,k}\left( J_{F_1}(0) \right)<r$  then  $J_{F_1}(0)$ is ill-conditioned (see Remark~\ref{cond}) and, in order to avoid numerical instabilities in the computations, no iterations are performed. 
Otherwise, that is if $\rank_{\delta,k}\left( J_{F_1}(0) \right)=r$ the matrix $ J_{F_1}(0) $ is well-conditioned; in this case, starting from the zero vector, the Root Finding Algorithm processes  the system $F_1=0$ by mean  the  NFA with an additional check, at each step, on the conditioning of the Jacobian matrix,  in order to  deal  with well conditioned linear systems.  When the loop stops, the algorithm returns the final iterate if it satisfies the constraints or, otherwise, a warning message. 

\begin{algorithm}{\bf (The Root Finding Algorithm -- RFA)}\label{algRF}\\
Let $D \subseteq \R^n$ be a set containing the origin and $F: D \rightarrow \R^m$ with $m<n$ be a  nonlinear function of class~$\mathcal C^2(D)$; let $\omega \ll 1$ be a fixed threshold, $\varepsilon>0$ and
$Q_\varepsilon=\{x \in \R^n \: : \: \|x\|_\infty \le
\varepsilon\}$; let  $\delta \ge \varepsilon$ of the same order of
magnitude as~$\varepsilon$ and $k>1$. 
Consider the following sequence of steps.
\begin{enumerate}
\item[{\bf RF1}] Compute the numerical $(\delta,k)$-rank of $(J_F(0)
\:|\: F(0))^t$ and a permutation matrix $\Pi \in \Mat_m(\R)$ as given
by Theorem~\ref{theoremRRQR} w.r.t.~the numerical rank$r=\rank_{\delta,k}\left ((J_F(0)
\:|\: F(0))^t\right)$. Partition $\Pi$ as $(\Pi_1 \:| \: \Pi_2)$ with
$\Pi_1 \in \Mat_{m \times r}(\R)$; let $F_1=\Pi_1^t F$.

\item[{\bf RF2}] Let $\bar x=0$, $h=(1,\ldots,1)^t$.\\
While ($\|h\|_2 > \omega$) and ($\rank_{\delta,k} (J_{F_1}(\bar x))=r$)\\
\phantom{aa} $\bullet$ Compute the minimal $2$-norm solution $h$ of  $J_{F_1}(\bar x) h = -F_1(\bar x)$.\\
\phantom{aa} $\bullet$ Let $\bar x=\bar x+h$.

\item[{\bf RF3}] If $\bar x \in Q_\varepsilon  \cap D $ return $\bar x$; otherwise return $\bar x=$NULL.

\end{enumerate}
\end{algorithm}

Step~{\bf RF1} requires the determination of the permutation matrix $\Pi$; the existence of $\Pi$ is guaranteed by
Theorem~\ref{theoremRRQR}, its computation can be performed using a Rank-Revealing QR Factorization \cite{HP92}. We refer to~\cite{BQ1},~\cite{BQ2} for a detailed description of algorithms for computing such a decomposition.

Note that, if $ J_{F_1}(0) $ is well-conditioned then there exists a closed neighbourough $\Delta_0$ of the origin contained in $Q_\varepsilon  \cap D$ such that $ J_{F_1}(x) $ is well-conditioned  for each $x \in \Delta_0$. In fact, since $F_1 \in \mathcal C^2(D)$, then $J_{F_1}(x)$ is a Lipschitz  function (with constant $\gamma$) on $\Delta_0$ and so, from Proposition~\ref{lemmaSigma}, for each $x \in \Delta_0$ we have that 
 $
\left | \sigma_r(J_{F_1}(x)) - \sigma_r(J_{F_1}(0))  \right |  \le \|J_{F_1}(x) - J_{F_1}(0)\|_2 \le
\gamma \|x\|_2 $
which implies $\sigma_r(J_{F_1}(x)) \ge  \sigma_r(J_{F_1}(0)) - \gamma \|x\|_2$.
Since $\sigma_r(J_{F_1}(0))> k\delta $, if~$\| x\|_2$ is small enough, that is if $x$ belongs to a suitable neighbourough
$\Delta_0$ of the origin, then $J_{F_1}(x)$ has full numerical $(\delta,k)$-rank on $\Delta_0$ and so the RFA executes some iterations  at step~{\bf RF2} for approximating the solution of $F_1=0$. 

We analyze the output of the RFA; if the output $\bar x \neq$ NULL   then the peculiarities of $\bar x$ give us different information on the underdetermined nonlinear system $F=0$.
 If $\bar x=0$ is returned, then two possible cases can occur:
either a single iteration or no iteration of the loop of step~{\bf RF2} has been performed. 
In the former case $h$ is  the zero vector: it means that $F_1(\bar x)=0$  and so $\widehat F(\bar x)=0$, that is the exact solution of $\widehat F=0$ has been found. From Theorem~\ref{Differenza_F} it follows that $\|F(\bar x)\|_2=\| F(\bar x)- \widehat F(\bar x)\|_2 < q(m,r)\delta$. 
 In the latter case no iteration has been executed since $\rank_{\delta,k}(J_{F_1}(0)) <r$, though 
$\rank_{\delta,k}(J_F(0)\:|\: F(0))=r$; in this situation $h=(1\dots 1)^t$ and, if  some simple hypotheses on a closed convex set
$\Delta \subseteq Q_\varepsilon \cap D$  are satisfied,  Theorem~\ref{noExistenceWellCondCase}
shows  that there are no solutions of $F=0$ in~$\Delta$.

\begin{theorem}\label{noExistenceWellCondCase}
Let $\Delta  \subseteq Q_\varepsilon \cap D$ be a closed convex set containing the origin and let $\gamma$ be the Lipschitz constant of  $J_{F_1}$  on~$\Delta$. Suppose  that the matrix
$\left( J_{F_1}(0) \:|\:  F_1(0)\right)$ has full numerical
$(\delta,k)$-rank $r$ with $k> 1+(\sqrt{n} + n \gamma)\delta$. If
$\rank_{\delta,k}\left (J_{F_1}(0) \right)< r$ then there are no
solutions of $F=0$ in $\Delta$.
\end{theorem}
\begin{proof}
Note that, since $F\in \mathcal C^2(\Delta)$ and $\Delta$ is a
closed convex set, $J_F$ is a Lipschitz function on~$\Delta$. Let
$r_1=\rank_{\delta,k}\left (J_{F_1}(0) \right)$ and~$A \in \Mat_{r
\times n}(\R)$ be such that $\rank(A)=r_1$ and $\| J_{F_1}(0) - A
\|_2 < \delta$: the existence of such a matrix~$A$ follows from
Lemma~\ref{consequenceNumRank}. Since~$A$ has exact rank~$r_1$
there exist a permutation matrix $\Pi \in \Mat_{r}(\R)$ and a
matrix $W \in \Mat_{(r-r_1) \times r_1}(\R)$ such that
\begin{equation*}
\Pi A= \left (\begin{array}{l} A_1 \\ W A_1
\end{array} \right )
\begin{array}{l} \}r_1 \\ \} r-r_1
\end{array} \;\;\; \textrm{and}\;\;\; \rank(A_1)=r_1
\end{equation*}
We can express the matrix $\Pi J_{F_1}(0)$ as follows
\begin{equation*}
\Pi J_{F_1}(0) = \left ( \begin{array} {c} A_1 +E_{11} \\ W
A_1+E_{21} \end{array} \right) \quad \textrm{with} \;\; E=
 \left( \begin{array} {c}E_{11} \\ E_{21}
\end{array} \right) \;\; \textrm{and} \;\; \|E\|_2  < \delta
\end{equation*}
Suppose for a  contradiction that there exists $x^* \in \Delta$
such that ${F(x^*)=0}$, and so $F_1(x^*)=0$. From Taylor's theorem
applied to~$F_1(x)$ we know that there exists a point~$z$ on the
line connecting~$0$ to~$x^*$, such that
$0=F_1(x^*)=F_1(0)+J_{F_1}(z)x^*$, and so
$$0=\Pi F_1(0) + \Pi J_{F_1}(z)x^*= \Pi F_1(0)+ \Pi J_{F_1}(0)x^* +\Pi (J_{F_1}(z)-J_{F_1}(0))x^*$$
Using an obvious partitioning $(F_{11},F_{12})^t$ and
$(E_{12},E_{22})^t$ of the matrices $\Pi F_1(0)$ and
${\Pi(J_{F_1}(z)-J_{F_1}(0))}$ we obtain
\begin{equation*}
0= \left (\begin{array}{c} F_{11}\\ F_{12}
\end{array} \right ) +
\left (\begin{array}{c} A_1 + E_{11}\\ W A_1 + E_{21}
\end{array} \right )  x^*
+\left( \begin{array}{c}E_{12}\\E_{22} \end{array}\right) x^*
\end{equation*}
and therefore
\begin{equation*}
\left \{ \begin{array}{lll}
A_1 x^* &=& -F_{11}  -(E_{11} +E_{12})x^* \\
F_{12} &=&  W\left(F_{11}  + (E_{11}  + E_{12})x^*\right) -(E_{21}
+ E_{22})x^*
\end{array} \right .
\end{equation*}
It follows that $\Pi \left ( J_{F_1}(0) \:|\: F_1(0) \right) $ is equal to the matrix 
\begin{equation*}
\left (
\begin{array}{cc}
 A_1 + E_{11}   & F_{11} \\
 WA_1  +E_{21} &W\left(F_{11}  + (E_{11}  + E_{12})x^*\right) -(E_{21} + E_{22})x^*
\end{array} \right )
\end{equation*}
Let $B \in \Mat_{r \times n}(\R)$ such that
\begin{equation*}
\Pi B = \left (
\begin{array}{ccc}
A_1 &\;\;& F_{11}  + (E_{11} +E_{12})x^*   \\
WA_1  & \;\;&W\left(F_{11}  + (E_{11} +E_{12})x^*\right)
\end{array} \right )
\end{equation*}
since $\|\cdot\|_2$ is invariant under permutation and $x^*, z \in
\Delta$ we have
\begin{equation*}
\begin{array}{rcl}
\left \| B - \left ( J_{F_1}(0) \:|\: F_1(0) \right ) \right \|_2
&=&
\left \| \left( -E \:|\:E x^* +  \Pi(J_{F_1}(z)-J_{F_1}(0))x^*\right) \right \|_2 \\
&\le&  \left \| E \right \|_2 \left ( 1+ \|x^*\|_2 \right )+ \|J_{F_1}(z)-J_{F_1}(0)\|_2 \|x^*\|_2 \\
&\le&  \delta (1+ \sqrt{n}\varepsilon )+ \gamma n\varepsilon^2 \le
\delta + ( \sqrt{n} + n \gamma)\delta^2
\end{array}
\end{equation*}
From $\rank(A_1) = r_1$ we have that $\rank(B)=r_1$, and so
$\sigma_r(B)=0$. From Proposition~\ref{lemmaSigma} it follows that
$\sigma_r\left ( J_{F_1}(0) \:|\: F_1(0) \right ) \le  \delta + (
\sqrt{n} + n \gamma)\delta^2$.
Since from the hypothesis $\sigma_r ( J_{F_1}(0) \:|\: F_1(0) ) >
k \delta$ and $k > 1 + ( \sqrt{n} + n \gamma)\delta$ we obtain a
contradiction.
\end{proof}

If the RFA returns a non zero vector $\bar x \in Q_\varepsilon \cap D$ then $F(\bar x)$ assumes small values; in particular, $\|F(\bar x)\|_2$ is bounded by a function depending on the data tolerance as  the following theorem shows.

\begin{theorem}\label{theoremValoreRho}
Let $\bar x$ be the output of the RFA applied to $F \in \mathcal
C^2(D)$ and $H$ be the closed convex hull of the sequence of
points computed by the RFA. If $H \subseteq Q_\varepsilon  \cap D$
then
$$
 \|\widehat F(\bar x)\|_2 = O(\delta^2) \quad \textrm{ and }\quad
 \|F(\bar x)\|_2 \le q(m,r)\delta +O(\delta^2)
$$
Further, let $\Delta$ be a closed convex set such that $H \subseteq \Delta \subseteq Q_\varepsilon  \cap D$; if there exists
$x^* \in \Delta$ such that $F(x^*)=0$ then $ \|F(\bar x)\|_2 = O(\delta^2)$.
\end{theorem}
\begin{proof}
Let $h$ be the displacement computed in the last iteration of the RFA, that is let $h$ satisfies $J_{F_1} (\bar x - h) h = -F_1(\bar
x - h)$ and consequently $J_{\widehat F} (\bar x - h) h = -\widehat F(\bar x - h)$. Since $\bar x-h \in H$,  $\widehat F\in
\mathcal C^2(H)$ and $H$ is a closed convex set, we express $\widehat F(\bar x)$ using the Taylor expansion of~$\widehat F$
centered at $\bar x- h$ we have
$\widehat F(\bar x ) = \widehat F(\bar x -h ) + J_{\widehat F}(\bar x- h ) h +O(\|h\|^2) = O(\|h\|^2)$
and, from  $\|h\|_2 \le \sqrt{n} \varepsilon \le \sqrt{n} \delta$, it  follows that $\|\widehat F(\bar x)\|_2= O(\delta^2)$.

Since $\|F(\bar x)\|_2 = \|\Pi F(\bar x)\|_2 \le \|\widehat F(\bar x) - \Pi F(\bar x)\|_2+ \|\widehat F(\bar x) \|_2$, from
Theorem~\ref{Differenza_F}  we have $\|F(\bar x)\|_2 \le q(m,r)\delta +O(\delta^2)$ and, if there exists $x^* \in \Delta$
such that $F(x^*)=0$ then $ \|F(\bar x)\|_2 = O(\delta^2)$.
\end{proof}

\section{The Low-degree Polynomial Algorithm}\label{mainResult}
Given a set~$\Xeps$ of $s$ distinct empirical points of $\R^n$ in this paper we address the problem to determine a 
polynomial of $P=\R[x_1,\ldots,x_n]$ whose degree is bounded by the minimal degree of the elements of the vanishing ideal $\I(\X)$ 
and whose affine variety exactly contains  an admissible perturbation~$\X^*$ of~$\Xeps$.

In this section we present an algorithm that computes an admissible perturbation~$\overline \X$ of~$\Xeps$ and a polynomial~$f$, whose degree is bounded by the minimal degree of the elements of $\I(\X)$; since $f$ has low degree and, by construction, assumes small values at $\overline \X$, it is a good candidate to be the solution to the addressed problem.
In fact, if some simple additional hypotheses are satisfied, $f$ turns out to be  the polynomial we are looking for, as in Theorem~\ref{theoremBasedOnKantorovich} we prove that there exists an admissible perturbation~$\X^*$, slightly differing from~$\overline \X$,  contained in 
the affine variety $\Z(f)$.

As already pointed out in the Introduction a straightforward application of the Buch\-ber\-ger-M\"oller (BM) algorithm to empirical data provides unreliable results, because the problem of deciding whether a polynomial vanishes at the input points is very sensitive to perturbations.
Nevertheless the strategy of the BM algorithm may be adapted to limited precision data; for instance the NBM and SOI algorithms (\cite{AFT08}, \cite{Fa10}) employ such an approach to compute polynomials which almost vanish at the input points.
Also in this context we present a modified version of the BM algorithm.

The core of the BM algorithm is the stepwise construction of the monomial basis~$\QB$ of the quotient ring~$P/\I(\X)$ viewed as an $\R$-vector space.
Initially~$\QB$ comprises just the power product~$1$;  at the generic iteration $\QB=\{t_1,\ldots,t_\nu\}$ and, fixing a term ordering $\tau$, the smallest $t >_\tau t_i$, $i=1\dots\nu$,  is considered. If the evaluation vector $t(\X)$ is linearly dependent on the columns of the matrix $M_\QB(\X)$, that is if
\begin{equation}\label{BMcheck}
\rho(\X)=0
\end{equation}
where $\rho(\X)=t(\X) - M_\QB(\X) \alpha(\X)$ is the residual of the least squares problem $M_\QB(\X) \alpha(\X) = t(\X)$, then
the polynomial $t-\sum_{i=1}^\nu \alpha_i(\X) t_i$ is formed and added to a basis, otherwise~$t$ is added to~$\QB$.

We adapt the generic iteration of the BM algorithm to our case by replacing check~(\ref{BMcheck})~by
\begin{equation}\label{problem}
 \rho(\X^*)=0 \textrm{ for some admissible perturbation} \; \X^* \; \textrm{of } \Xeps
\end{equation}
If problem~(\ref{problem}) admits solution then the polynomial $f^*=t-\sum_{i=1}^\nu \alpha_i(\X^*) t_i$ is formed; $f^*$  and $\X^*$ are a solution we were looking for (by construction~$f^*$ vanishes at $\X^*$). Otherwise~$t$ is added to~$\QB$.

Analogously to the strategy used in \cite{AFT08}, in order to make problem~(\ref{problem}) effectively
 solvable we express any admissible perturbation of~$\Xeps$ as a function of~$sn$ 
 {\bf error variables} $\ebold = (e_{11},\dots,e_{s1},e_{12},\dots,e_{s2},\dots,e_{1n},\dots,e_{sn})$, 
 where each variable~$e_{kj}$ represents the perturbation in the~$j$-th coordinate of the specified 
 value $p_k \in \X$. Specifically, a generic admissible perturbation~$\Xtilde$ of~$\Xeps$ can be 
 expressed as $\Xtilde=\X(\ebold)=\{p_1(\ebold_1),\ldots,p_s(\ebold_s)\}$ where for each $k= 1 \ldots s$
\begin{equation*}
\begin{array}{lll}
&&\ebold_k = (e_{k1},\ldots,e_{kn})\\ 
&&p_k(\ebold_k) = \left (
p_{k1} + e_{k1},\; p_{k2} + e_{k2}, \dots p_{kn}  +  e_{kn} \right) \;\;\; \textrm{ and } \;\;\; \|\ebold_k\|_\infty \le \varepsilon
\end{array}
\end{equation*}
Using this formalization,  $\Xtilde=\X(\ebold)$ is a function in the variables~$\ebold$ with admissible domain $Q_\varepsilon=\{\ebold \in\R^{sn} \: : \:  \|\ebold\|_\infty \le \varepsilon\}$. It follows that $\alpha(\Xtilde)$, $\rho(\Xtilde)$, and $M_\QB(\Xtilde)$ are functions of $\ebold$, simply denoted by $\alpha(\ebold)$, $\rho(\ebold)$ and $M_\QB(\ebold)$; in particular $\rho:\R^{sn} \rightarrow \R^s$ is defined by
\begin{equation}\label{rho}
\rho(\ebold)=t(\ebold)-M_\QB(\ebold)\alpha(\ebold)
\end{equation}
where $\alpha(\ebold)=M_\QB(\ebold)^\dagger t(\ebold)$.
It follows that problem ($\ref{problem}$) is equivalent to determine whether
\begin{equation}\label{rhoE}
\rho(\ebold)=0   \textrm{ for some }  \ebold \in Q_\varepsilon \cap  D_\QB
\end{equation}
where $D_\QB$ is the domain of $\rho$, that is $D_\QB=\{\ebold \in \R^{sn} : M_\QB(\ebold) \textrm{ is full rank}\}$.

In~\cite{Fa10} it is proved that the problem~(\ref{rhoE}) has no solution if
\begin{equation}\label{condition}
|\rho(0)| > \varepsilon | I - M_\QB(0) M_\QB(0)^\dagger| \sum_{k=1}^n |\partial_k t(0) - M_{\partial_k \QB}(0) \alpha(0)| + O(\varepsilon^2)
\end{equation}
where  the absolute value of a matrix is  the matrix consisting of the absolute values of its elements and the lower bound  means that the relation holds componentwise.
In our algorithm we employ such a result to study problem~(\ref{rhoE}): we solve the underdetermined nonlinear system~(\ref{rhoE}) only if condition~(\ref{condition}) is not satisfied.
Unfortunately, finding an exact solution of problem~(\ref{rhoE}) can be computationally very hard: in fact the analytic expression of~$\rho(\ebold)$ can be very hard to obtain and even when an explicit expression of~$\rho(\ebold)$ is available, an exact solution $\ebold^*$ of $\rho(\ebold)=0$ can be computed only in very few cases (specifically when~$n$ and~$s$ are small), while in general only an approximation of $\ebold^*$ can be effectively determined. In our algorithm we approximate the solution of problem~(\ref{rhoE}) using the results of Section~\ref{underdetNonlinearSystem}; in particular,  we design the The Low-degree Polynomial Algorithm following the frame of the BM algorithm and solving the nonlinear systems~(\ref{rhoE}) using the RFA.

\begin{algorithm}\label{newNBM}{\bf (The Low-degree Polynomial Algorithm - LPA)}\\
Let $\Xeps=\{p_1^\varepsilon,\ldots,p_s^\varepsilon\}$ be a finite set of distinct empirical points of~$\R^n$, $k>1$, $\delta \ge \varepsilon$ of the same magnitude as~$\varepsilon$ and $\tau$ be a term ordering.\\
Consider the following sequence of steps.
\begin{enumerate}
\item[{\bf LP1}] Start with  $t=1$ and the lists $\QB=[1]$ and $L=[\;]$.
\item[{\bf LP2}] Add to~$L$ those elements of $\{x_1 t, \ldots,x_n t\}$ which are not multiples of elements of~$L$.  Let $t=\min_{\tau}(L)$ and delete it from $L$.
\item[{\bf LP3}] Choose $\bar \ebold=0$ and consider the residual function $\rho$ defined in~(\ref{rho}).
\item[{\bf LP4}] If $\rho(0)$ satisfies condition~(\ref{condition}) then\\
$\phantom{xxx}$ $\bullet$ add~$t$ to the list~$\QB$ and continue with step ${\bf LP2}$;\\
\phantom{aaa} else\\
$\phantom{xxx}$ $\bullet$ 
apply the RFA to $\rho$ on $Q_\varepsilon \cap D_\QB$ with thresholds~$k$ and $\delta$.
\item[{\bf LP5}] If  the RFA returns NULL  or $\bar \ebold=0$ and  $\|h\| \ne 0$  then\\
$\phantom{xxx}$ $\bullet$ add~$t$ to the list~$\QB$ and continue with step~${\bf LP2}$;\\
\phantom{aaa} else\\
$\phantom{xxx}$ $\bullet$ compute $f=t-\sum_{t_i \in \QB} \alpha_i(\bar \ebold)t_i$; return  $\bar \ebold$, $f$, $\QB$ and stop.
\end{enumerate}
\end{algorithm}
We make a few observations on the LPA.
At each iteration the domain~$D_\QB$ of the residual function $\rho$ contains the origin. In fact since the residuals computed at the previous iterations do not vanish at zero the matrix $M_\QB(0)$ is full rank.
Further, at step ${\bf LP4}$ each iteration of the RFA requires the evaluation of~$\rho$ and~$J_\rho$ at a given point~$\ebold$.
The value of $\rho(\ebold)$ is simply obtained by solving the numerical least squares problem $M_\QB(\ebold) \alpha(\ebold)= t(\ebold)$ and computing its residual.
Proposition \ref{Calcolo_Jac} illustrates how to compute each column of the Jacobian matrix $J_\rho(\ebold)$ without passing through the explicit expression of the function~$\rho$.

\begin{theorem}
The Low-degree Polynomial Algorithm stops after finitely many steps and returns $\bar \ebold$, $f$, $\QB$ such that $\bar \ebold \in Q_\varepsilon \cap D_\QB$.

Let $\rho$ be the residual function computed during the last iteration of the LPA and $H$ be the closed convex hull of the sequence of points computed by the RFA applied to~$\rho$. If $H \subseteq Q_\varepsilon  \cap D_\QB$ then
$$\|f(\X(\bar \ebold))\|_2 \le q(s,r)\delta + O(\delta^2) $$
where $r=\rank_{\delta,k}(J_\rho(0)\:|\: \rho(0))$ and $q(s,r)=\sqrt{r(s-r) + \min(r,s-r)}$.

Further, let $\Delta$ be a closed convex set such that $H \subseteq \Delta \subseteq Q_\varepsilon  \cap D_\QB$;
if there exists $\ebold^* \in  \Delta$ such that $\rho(\ebold^*)=0$ then
$\|f(\X(\bar \ebold))\|_2 = O(\delta^2) $
\end{theorem}

\begin{proof}
At each iteration either the algorithm computes the polynomial $f$ and thus it stops, or a  term $t$ is added to $\QB$.
We observe that this latter instruction can be executed at most~$s-1$ times; in fact when~$\QB$ contains~$s$ terms, that is when $M_\QB(0)$ becomes a square matrix, the residual vector~$\rho(0)$ is zero, and  consequently $\bar \ebold =0$, $\rho(\bar \ebold)=0$, and a polynomial $f $ is computed. From the check of step ${\bf LP5}$ obviously follows that  $\bar \ebold \in Q_\varepsilon \cap D_\QB$.

The residual function $\rho \in \mathcal C^\infty(D_\QB)$ since each of its components is a rational function defined in $D_\QB$. The  thesis follows from Theorem~\ref{theoremValoreRho} applied to $\rho$ and the equality $f(\X(\bar \ebold))=\rho(\bar \ebold)$.
\end{proof}

In the following theorem, under additional hypotheses on the polynomial $f$, output of the LPA, we prove that there exists an admissible perturbation~$\X(\ebold^*)$ of~$\Xeps$ contained in $\Z(f)$;
 further, we provide an estimation of the distance between~$\X(\bar \ebold)$ and~$\X(\ebold^*)$
 which is usually much smaller than $\varepsilon$.

\begin{theorem}\label{theoremBasedOnKantorovich}
Let $(\bar \ebold, f, \QB)$ be the output of the LPA applied
to~$\Xeps$. For each $i=1 \ldots s$ let $f_i: \R^n \rightarrow \R$ be the function
defined by $f_i(\ebold_i)= f(p_i(\ebold_i))$; let  $r_i$ and $R_i$ be real positive numbers such that
\begin{equation*}
r_i=\varepsilon-\|\bar \ebold_i \|_\infty \quad and \quad  R_i < \min
\left\{ r_i, \; \frac{\|J_{f_i}(\bar \ebold_i)\|_2}{\gamma_i} \right\}
\end{equation*}
where $\gamma_i>0$ is a Lipschitz constant of~$J_{f_i}$ in
$B(\bar \ebold_i,r_i)$. Let~$\mu_i$ be an upper bound of~$\|J_{
f_i}^\dagger\|_2$ in~$B(\bar \ebold_i,R_i)$. If $\|J_{f_i}(\bar \ebold_i)\|_2 >0 $
and
\begin{equation}\label{valueRho}
| f_i (\bar \ebold_i) | <
\frac{ R_i} {\mu_i(2+\gamma_i R_i\mu_i)} \equiv \chi_i
\;\;\;\;\;\;\;\;\;\;\;\;\;\; \forall i=1\dots s
\end{equation}
then there exists $\ebold^*$ such that the set $\X(\ebold^*)$ is an admissible
perturbation of~$\Xeps$, $f(\X(\ebold^*))=0$, and $\|\ebold^* - \bar \ebold\|_\infty < \max R_i$.
\end{theorem}

 \begin{proof}
Let ${\bar \alpha_j}=\alpha_j(\bar \ebold)$ and $f= t - \sum_{t_j \in \QB} {\bar \alpha_j} t_j$.
We observe that $f_i(\ebold_i)= t(p_i(\ebold_i)) +
\sum_{t_j \in \QB} {\bar \alpha_j} t_j(p_i(\ebold_i))$ is a polynomial function of $\ebold_i$ and its Jacobian~$J_{f_i} (\ebold_i) \in \Mat_{1\times n}(\R) $  is a Lipschitz function in $\bar B(\bar \ebold_i,r_i)$.
We prove that $\|J_{f_i}(\ebold_i)\|_2 >0$ in~$\bar B(\bar \ebold_i, R_i)$.
For each $\ebold_i \in \bar B(\bar \ebold_i,R_i)$ we~have
$$
\left | \; \|J_{f_i}(\ebold_i)\|_2 - \|J_{f_i}(\bar \ebold_i)\|_2  \; \right |  \le 
\|J_{f_i}(\ebold_i) - J_{f_i}(\bar \ebold_i)\|_2 \le 
\gamma_i \|\ebold_i - \bar \ebold_i\|_2 \le  \gamma_i R_i
$$
and so $\|J_{f_i}(\ebold_i)\|_2 \ge \|J_{f_i}(\bar \ebold_i)\|_2 - \gamma_i R_i > 0$.
It follows that $J_{f_i}(\ebold_i)$ has constant rank $1$
in~$\bar B(\bar \ebold_i, R_i)$; further,  $\|J^\dagger_{f_i} (\ebold_i)\|_2$ is upper bounded in $\bar B(\bar \ebold_i,R_i)$  since 
 $$
 \|J^\dagger_{f_i}(\ebold_i)\|_2 = \frac{\|J^t_{f_i}(\ebold_i)\|_2}{\|J_{f_i}(\ebold_i)\|_2^2} =
 \frac{1}{\|J_{f_i}(\ebold_i) \|_2} \le \frac{1}{\min_{\bar B(\bar \ebold_i,R_i)} \|J_{f_i}(\ebold_i)\|_2} 
  $$
 and $\min_{\bar B(\bar \ebold_i,R_i)} \|J_{f_i}(\ebold_i)\|_2>0$.
 
Now we apply Kantorovich theorem to $f_i$ on $B(\bar \ebold_i,R_i)$ with
$x_0=\bar \ebold_i$, $p=1$, $\eta=R_i$ and so $\Omega_\eta=\{\bar \ebold_i\}$; since
$\chi_i=\frac{ R_i}{\mu_i(2+\gamma_i R_i\mu_i)}$ satisfies
$\frac{\gamma_i \mu_i^2 \chi_i}{2}<1$ and  
$\frac{2\mu_i \chi_i}{(2-\gamma_i \mu_i^2 \chi_i)} < R_i$,
and~(\ref{valueRho}) holds we conclude there exists ${\ebold^*_i=(e^*_{i1} \dots e^*_{in})}$  in $B(\bar \ebold_i,R_i)$ such that $\bar f_i({\ebold^*_i})=0$.  
Let $\ebold^*=(e^*_{11}\dots e^*_{s1}, \dots, e^*_{1n}\dots e^*_{sn} )$; since $f(\X(\ebold^*))=(f_1(\ebold^*_1), \ldots,f_s(\ebold^*_s))^t$ we have $f(\X(\ebold^*))=0$. Furthermore, since the $i$-th element of  $\X$ is $p_i=p_i(0)$ and 
$$\|p_i(\ebold^*_i) -p_i(0)\|_\infty = \|\ebold^*_i \|_\infty  \le  \|\ebold^*_i -\bar \ebold_i\|_2+ \|\bar \ebold_i \|_\infty \le  R_i +\|\bar \ebold_i \|_\infty < \varepsilon$$
we have that $\X(\ebold^*)$ is an admissible perturbation of $\Xeps$. Finally, we have that
$\|\ebold^* - \bar \ebold\|_\infty = \max_i \|\ebold^*_i - \bar \ebold_i\|_\infty \le \max_i \|\ebold^*_i - \bar \ebold_i\|_2 < \max_i R_i < \varepsilon$ which concludes the proof.
\end{proof}

\section{Numerical examples}\label{examples}
In this section we present some numerical examples to show the effectiveness of the~LPA.
We implemented the LPA and the RFA using the C++ language, the CoCoALib \cite{Co}, and some routines
of GSL - GNU Scientific Library \cite{GSL}; all computations have been performed on an
Intel Core 2 Duo processor (at 1.86 GHz).
In all the examples the degree lexicographic term ordering with $y<x$ is understood; furthermose the coordinates of the points and the coefficients of the polynomials are sometimes displayed as truncated decimals. 

In Example~\ref{primo} the LPA is applied to the set of points of Example~\ref{mainExample} created by perturbing by less than $0.1$ the coordinates of $10$ points located on the parabola $g=0$ where
$g=y^2-x-2y+2$.
In contrast to the exact approach that provides a minimal degree polynomial of degree $4$, the LPA detects the geometrical configuration given by a parabola. 

\begin{example}\label{primo}
Let $\X \subset \R^2$ be the set of points given in Example~\ref{mainExample},
let $\varepsilon=0.1$, $\delta=2 \varepsilon$ and $k=2$. 
The LPA applied to~$\X^\varepsilon$ performs three iterations.
At the first two iterations the terms $t=y$ and $t=x$ are analyzed: since in both cases condition~(\ref{condition}) is satisfied the set $\QB=\{1,y,x\}$ is constructed. During the third iteration the term $t=y^2$ is considered: at step~{\bf LP4} the RFA applied to the current residual function~$\rho$ returns $\bar \ebold \in Q_\varepsilon \cap D_\QB$, therefore at step~{\bf LP5} a polynomial $f$ is computed (see Figure~\ref{Fig11}), and the algorithm stops with the following output:
\begin{equation*}
\begin{array}{rcl}
 \bar \ebold  &=& 10^{-1} ( 0.474,  -0.211,  -0.146,  -0.100,   0.075,  -0.105,-0.091, -0.017,
 \\ &&   0.066,   0.054,  0.004, -0.215, 0.285, -0.202, 0.211, 0.378, -0.332,\\ 
 &&     -0.065,  -0.369,   0.305)\\
f &=& y^2 -0.9751012065x -2.0049270587y + 1.9775224038\\
\QB &=& \{1,y,x\}
\end{array}
\end{equation*}

\begin{figure}[htb]
\centering
\begin{minipage}[c]{0.46\textwidth}
\includegraphics[width=\textwidth]{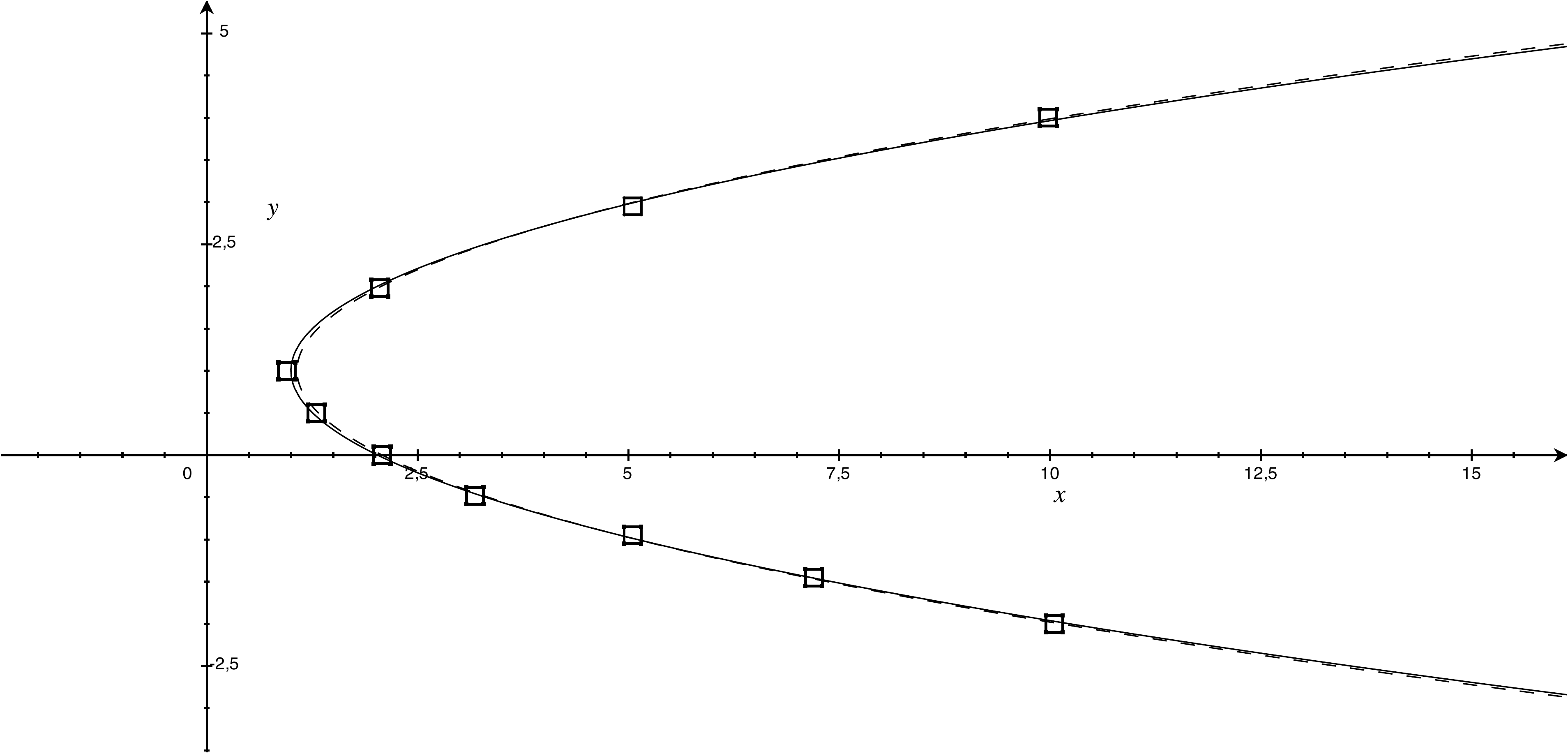}
\vspace{-.5cm}
\caption{Variety $\Z(f)$}\label{Fig11}
\end{minipage}%
\hspace{0.8cm}
\begin{minipage}[c]{0.46\textwidth}
\includegraphics[width=\textwidth]{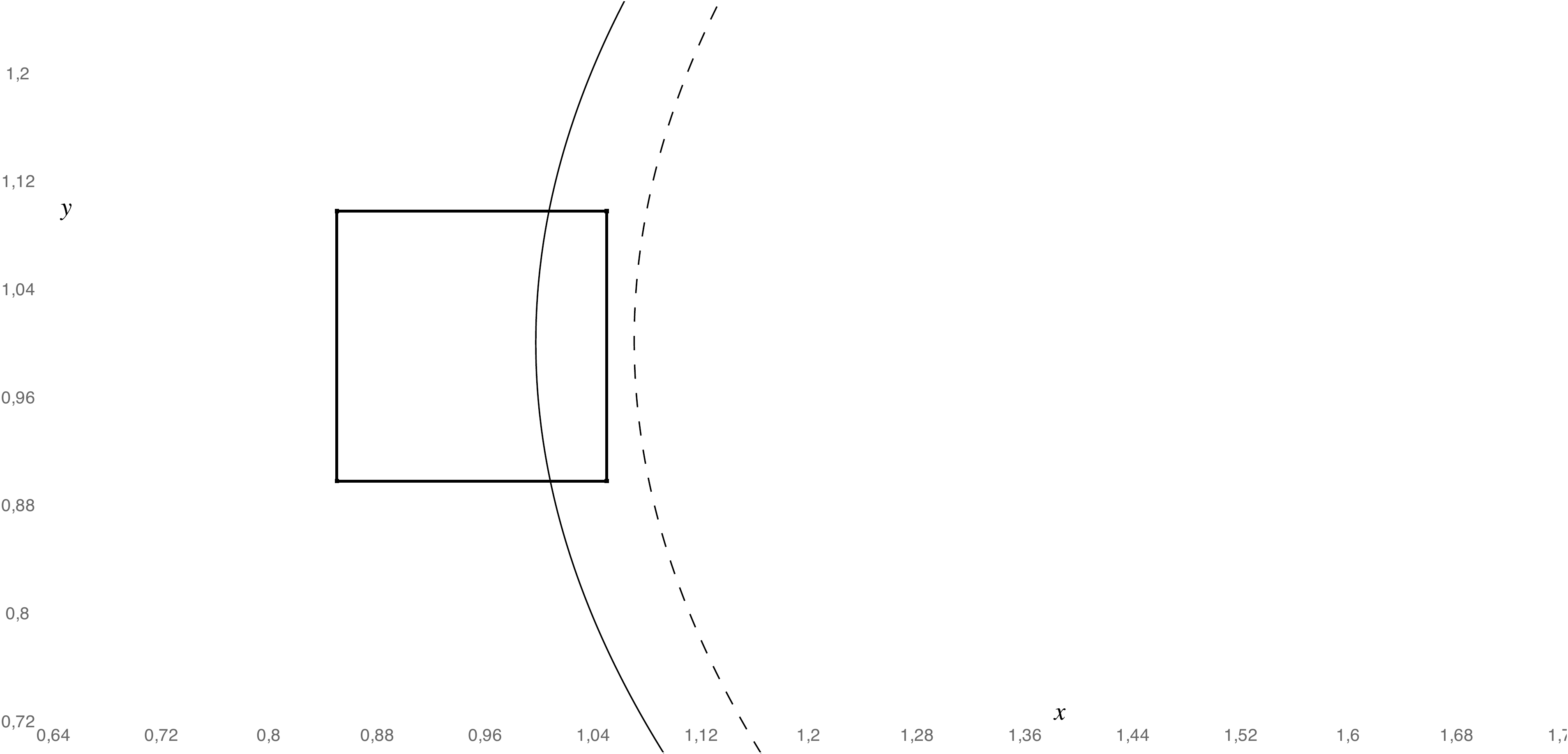}
\vspace{-.5cm}
\caption{The varieties $\Z(f)$~and $\Z(h)$: zoom~in~around the point (0.95, 1)}\label{Fig12}
\end{minipage}
\end{figure}

The set $\X(\bar \ebold) $ is an admissible perturbation
of~$\Xeps$; we can consider the polynomial $f$ almost
vanishing at~$\X(\bar \ebold)$ since  $\|f(\X(\bar \ebold))\|_2 /\|f\| < 2 \cdot 10^{-5} $.
We apply Theorem~\ref{theoremBasedOnKantorovich} to this example and report the values of $R_i$, $\chi_i$
and $|f_i(\bar \ebold_i)|$ in the following table. Since for each $i$ the inequality (\ref{valueRho}) is satisfied, we conclude that there exists $\ebold^*$
such that $\X(\ebold^*)$ is an admissible perturbation of~$\Xeps$, $f(\X(\ebold^*))=0$
and $\|\ebold^* - \bar\ebold\|_\infty < 0.0935$. 
\begin{table}[h]
\begin{center}
  \begin{tabular}{| c|c|c|c|c|c| }
    \hline
           & i=1 & i=2& i=3 & i=4 & i=5 \\ \hline
    $R_i$    &  0.0526 &   0.0785 &   0.0715  &  0.0798  &  0.0789   \\  \hline
    $\chi_i$ &  0.0243 &   0.0354 &   0.0325  &  0.0360  &  0.0356   \\  \hline
    $|f_i(\bar \ebold_i)|$    & $3.3 \cdot 10^{-7}$ & $3.87\cdot 10^{-6}$  &  $3.46\cdot 10^{-6}$  &  $4.77 \cdot 10^{-6}$ &
  $ 1.46\cdot 10^{-6} $\\ \hline \hline
           & i=6 & i=7& i=8 & i=9 & i=10 \\ \hline
    $R_i$    & 0.0622 &   0.0668   & 0.0935  &  0.0631  &  0.0695 \\ \hline
    $\chi_i$ & 0.0285 &   0.0305   & 0.0416  &  0.0289  &  0.0316 \\ \hline
    $|f_i(\bar \ebold_i)|$    &   $  9.23\cdot 10^{-6} $ & $ 4.45\cdot 10^{-6} $ & $ 1.015 \cdot 10^{-5}$  & $2.783\cdot 10^{-5}$  &$ 4.95\cdot 10^{-6}$  \\ \hline
  \end{tabular}
  \end{center}
  \end{table}
 
The NBM algorithm as well as the SOI algorithm (available in~\cite{Co}) applied to~$\Xeps$ detect the same support of~$f$, since they compute the polynomial $h=y^2-1.0041x-2.0089y+2.1287$ which almost vanishes at~$\X$.  Nevertheless~$h$ is not a solution to our problem: in fact, $h=0$ (the dashed line in Figure~\ref{Fig12}) contains no admissible perturbation of $\Xeps$ since it does not cross the $\varepsilon$-neighborhood of the first point of~$\X$.
\end{example}

In the following example a fairly large set of points of~$\R^5$ is considered.
\begin{example}
Let $\X \subset \R^5$ be a set of points created by perturbing by less than~$0.1$ the coordinates
of $25$ points lying on the affine variety $\Z(g)$ where
\begin{equation*}
\begin{array}{lll}
g&=&x_4^2 -\frac{17}{41}x_4x_5-2x_5^2 -\frac{10}{41}x_1 +\frac{21}{41}x_2-\frac{74}{41}x_3+ \frac{93}{41}x_4 -\frac{36}{41}x_5+ \frac{39}{41}
\end{array}
\end{equation*}
Let $\varepsilon=0.1$, $\delta=2 \varepsilon$ and ${k=2}$; the LPA applied to~$\X^\varepsilon$ 
performs $11$ iterations in a computational time of~$24$ seconds.
During the $11$th iteration the term $t=x_4^2$ is considered and the output $\bar \ebold$, $f$, $\QB$ is returned, where $\|\bar \ebold\|_\infty = 0.0789$
\begin{equation*}
\begin{array}{lll}
f &=& x_4^2  -0.421x_4x_5-1.994x_5^2-0.244x_1+ 0.405x_2 -1.779x_3 +2.328x_4\\ 
&&  -0.685x_5 + 1.124
+10^{-2} \cdot (0.056 x_1x_5 -2.549x_2x_5 + 1.679  x_3x_5)\\
\QB &=& \{1, x_5, x_4, x_3, x_2, x_1, x_5^2, x_4x_5, x_3x_5, x_2x_5, x_1x_5\}
\end{array}
\end{equation*}
The set $\X(\bar \ebold) $ is an admissible perturbation
of~$\Xeps$; we can 
consider the polynomial $f$ almost vanishing at~$\X(\bar \ebold)$
since  $\|f(\X(\bar \ebold))\|_2 /\|f\| \approx 3.42 \cdot 10^{-4} $.
We apply Theorem~\ref{theoremBasedOnKantorovich} to this example; 
since for each $i$ the inequality (\ref{valueRho}) is satisfied, we conclude that there exists $\ebold^*$
such that $\X(\ebold^*)$ is an admissible perturbation of~$\Xeps$, $f(\X(\ebold^*))=0$
and $\|\ebold^* - \bar\ebold\|_\infty < 0.097$. 

We observe that the minimum of the degrees of the polynomials of~$\I(\X)$ is $3$, and this
does not suggest that the points of $\X$ lie close to the variety~$\Z(g)$.
On the contrary, since the coefficients of $f$ differ only slightly from the corresponding coefficients of~$g$, 
the LPA allows us to recover the ``approximate" geometrical configuration of $\X$.
\end{example}

The following example (taken from~\cite{CGKW00}) is different from the standard cases deriving from 
the analysis of real-world phenomena. 
It points out that the LPA has various possibilities of applications; in this case it is employed to obtain 
the numerical implicitization of a parametric curve.


\begin{example}\label{secondo}
We consider the parametric equations for a B\'ezier curve:
\begin{equation}\label{paramEq}
\begin{array}{l}
x(t)= 4t(2t^5 -3t^4 + 8t^2 + 6t +3)/(t^6-3t^5 + 3t^4+3t^2 +3t+1) \\
y(t) = 6t(4t^4+9t^3 -9t^2 - 9t +5)/(t^6-3t^5 + 3t^4+3t^2 +3t+1)
\end{array}
\end{equation}
whose implicit equation is $g=0$ where
\begin{equation*}
\begin{array}{lll}
g &=& x^3 - \frac{2}{1269}x^2y - \frac{28}{423}xy^2 + \frac{224}{34263}y^3 - \frac{15712}{1269}x^2 -\frac{56}{1269}xy + \frac{848}{3807}y^2 \\
&& + \frac{44480}{1269}x - \frac{17792}{1269}y 
\end{array}
\end{equation*}
We consider the set of points
\begin{eqnarray*}
\X &=& \{ (0,  0),  (-1.3581,   -4.7661), (2.0956,    2.0315), (4.6884,   -0.3349),\\
&& (-2.7205,  -11.6848), (-7.2835,  -40.9773),  (6.7793,   -1.4114),\\
&& (8.6024,    1.4575),(10.5250,    8.8937), (12.6213,   19.7217) \}
\end{eqnarray*}
created by evaluating the parametric equations~(\ref{paramEq}) at~$10$ random values of the parameter $t \in (-1,2)$,
and rounding off up to $10^{-4}$.

Let $\varepsilon=10^{-4}$, $\delta=2 \varepsilon$ and ${k=2}$; the LPA applied to~$\X^\varepsilon$ performs $9$ iterations.
During the first $8$ iterations 
condition~(\ref{condition}) is satisfied, and the set ${\QB=\{1,y, x, y^2, xy, x^2, y^3, xy^2, x^2y\}}$ is constructed.
During the $9$-th iteration the term $t=x^3$ is considered: at step~{\bf LP4} the RFA applied to the current residual function~$\rho$ returns $\bar \ebold_1 \in Q_\varepsilon \cap D_{\QB}$, therefore at step~{\bf LP5} a
polynomial~$f_1$ is computed, and the algorithm stops with the output:
\begin{equation*}
\begin{array}{rcl}
 \bar \ebold_1  &=& 10^{-6} \cdot (8.13,  -1.99,  0.92,  -4.90,  1.40,  -7.63,  -3.30,   23.79,-24.90,  \\
 &&   8.50, -3.25,  0.36,-0.13,  1.11, 4.73,  -7.19,  9.02,  -8.41, 5.32,  -1.54)\\
f_1 &=& x^3 -0.00159x^2y - 0.06618xy^2 +0.00654y^3 -12.3814x^2 -0.04396xy\\
&&  +0.22269y^2 +35.0514x -14.0206y -0.00033\\
\QB &=& \{1,y, x, y^2, xy, x^2, y^3, xy^2, x^2y\}
\end{array}
\end{equation*}

Note that the coefficients of $f_1$ and $g$ are very close to each other; 
nevertheless, an even better result can be obtained by taking into account the exact information that the B\'ezier curve $g=0$
vanishes at the origin. 
To this aim we impose that the constant term of the polynomial computed by the LPA is zero;
under this assumption and using the above values for the parameters $\varepsilon$, $\delta$, $k$, the LPA applied to~$\X^\varepsilon$
stops with $\bar \ebold_2$, $f_2$, $\QB$, where
\begin{equation*}
\begin{array}{rcl}
 \bar \ebold_2  &=& 10^{-6} \cdot (0,  -2.09,  0.96,  -5.13, 1.46, -8.00, -3.45,24.92, -26.09, 8.90, \\
&& 0, 0.37,-0.14, 1.16,4.95,  -7.53,  9.44,  -8.81,  5.57,  -1.61)\\ 
 f_2 &=& x^3  -0.00158x^2y -0.06619xy^2 +0.00653y^3 -12.3814x^2 -0.04407xy \\ 
 &&+0.22271y^2 +35.0513x -14.0205y
\end{array}
\end{equation*}
We observe that $f_2$ is better than $f_1$ because its coefficients provide an even closer approximation of the coefficients of $g$, and its support and the support of~$g$ are the same.
\end{example}

%

\section{Appendix}\label{appendix}
The following proposition illustrates how to compute each column of the Jacobian matrix $J_\rho(\ebold)$, that is the vector $\partial \rho (\ebold)/\partial e_{ki}$, without passing through the explicit expression of the function~$\rho$.

\begin{proposition}\label{Calcolo_Jac}
Let $\QB = \{ t_1 \ldots t_m\}$ be a finite set of terms, let $t$ be a term, $t \notin \QB$, and  let $\rho$ be given by~(\ref{rho}); then
\begin{eqnarray*}
\frac{\partial \rho(\ebold)}{\partial e_{ki}}=
\left [ I - M_\QB(\ebold)M_\QB^\dagger(\ebold)\right] 
\left [\frac{\partial t(\ebold)}{\partial e_{ki}} -
\frac{\partial M_\QB(\ebold)}{\partial e_{ki}} \alpha(\ebold) \right ] - 
M_\QB^\dagger(\ebold)^t \frac{\partial M_\QB^t(\ebold) }{\partial e_{ki}}\rho(\ebold)
\end{eqnarray*}
where $I \in \Mat_s(\R)$ is the identity matrix, $ M_\QB^\dagger(\ebold) =\left( M_\QB^t(\ebold)M_\QB(\ebold) \right)^{-1} M_\QB^t(\ebold)$, 
$\partial t(\ebold) / \partial e_{ki}$ is the zero vector except the $k$-th coordinate which is equal to $\partial_i t(p_k(\ebold_k))$ and $\partial M_\QB(\ebold) / \partial e_{ki} $ is the zero matrix except the  $k$-th row which is equal to  $\left(\partial_i t_1 (p_k(\ebold_k)) \dots \partial_i t_m (p_k(\ebold_k))\right)$.
\end{proposition}

\begin{proof}
From the equality $\rho(\ebold)=t(\ebold)-M_\QB(\ebold) \alpha(\ebold)$, deriving matrices and vectors componentwise w.r.t.~$ e_{ki}$, we obtain
\begin{equation} \label{der_ro}
\frac{\partial \rho(\ebold)}{\partial e_{ki}}=
\frac{\partial t(\ebold)}{\partial e_{ki}}-
\frac{\partial M_\QB(\ebold)}{\partial e_{ki}}
\alpha(\ebold) - M_\QB(\ebold) \frac{\partial
\alpha(\ebold)}{\partial e_{ki}}
\end{equation}
Deriving componentwise each side of $M_\QB^t(\ebold)  M_\QB(\ebold)  \alpha(\ebold)  =  M_\QB^t(\ebold) t(\ebold) $ w.r.t.~$ e_{ki}$ we have
\begin{equation*}
\frac{\partial \alpha(\ebold)}{\partial e_{ki}} =  \left( M_\QB^t(\ebold)M_\QB(\ebold) \right)^{-1}
\left[ \frac{\partial M_\QB^t(\ebold)}{\partial e_{ki}} \rho(\ebold) + M_\QB^t(\ebold) \left (\frac{\partial t(\ebold)}{\partial e_{ki}}
- \frac{\partial M_\QB(\ebold)}{\partial e_{ki}}  \alpha(\ebold) \right) \right]
\end{equation*}
and the thesis follows  by substituting the expressions of $\partial \alpha(\ebold)/\partial e_{ki}$~in~(\ref{der_ro}).

For describing  $\frac{\partial t(\ebold)}{\partial e_{ki}} $  and  $  \frac{\partial M_\QB(\ebold)}{\partial e_{ki}} $ 
we consider the derivative of a generic term
$q=\Pi_{i=1}^n  x_i^{\beta_i}$. Since
$ q(\ebold) = \left ( \Pi_{i=1}^n(p_{1i} + e_{1i})^{\beta_i}, \dots,\Pi_{i=1}^n(p_{si} + e_{si})^{\beta_i}  \right )^t$
the vector $\frac{\partial q(\ebold)}{\partial e_{ki}} $ is the zero vector except its $k$-th coordinate which is equal to $\partial_i q(p_k(\ebold_k))$. It follows that only the $k$-th coordinate of the vector $\frac{\partial t(\ebold)}{\partial e_{ki}} $  and the $k$-th row of the matrix $  \frac{\partial M_\QB(\ebold)}{\partial e_{ki}} $ can be different to zero: they are equal to $\partial_i t(p_k(\ebold_k))$ and  to  $\left (\partial_i t_1 (p_k(\ebold_k)) \dots \partial_i t_m (p_k(\ebold_k))\right)$ respectively.
\end{proof}

\section*{Acknowledgments}
The authors would like to thank Prof.~D.~Bini and Prof.~L.~Robbiano for their
constructive remarks and helpful suggestions, and Dr.~S. Pasquero for his useful comments about the manuscript.
%

\bibliographystyle{model1-num-names}
\bibliography{<your-bib-database>}

\begin{thebibliography}{00}
\bibitem{AFT07}
J.~Abbott, C.~Fassino, and M.~Torrente,
\textit{Thinning Out Redundant Empirical Data.}
Math. Comput. Sci. 1 (2007), no. 2, 375--392. 
%
\bibitem{AFT08}
J.~Abbott, C.~Fassino, and M.~Torrente,
\textit{Stable Border Bases for Ideals of Points.}
J. Symbolic Comput. 43 (2008), no. 12, 883--894.
%
\bibitem{BQ1}
C.~H. Bischof and G.~Quintana-Ort\`i,
\textit{Computing Rank-Revealing QR Factorizations of Dense Matrices.}
ACM Trans. Math. Software 24 (1998), no. 2,  226--253.
%
\bibitem{BQ2}
C.~H.~Bischof and G.~Quintana-Ort\`i,
\textit{Algorithm 782: Codes for Rank-Revealing QR Factorizations of Dense Matrices.}
ACM Trans. Math. Software 24 (1998), no. 2,  254--257.
%
\bibitem{BM82}
B.~Buchberger and  H.~M.~M\"oller,
\textit{The construction of multivariate polynomials with preassigned zeros.}
{Proc. EUROCAM '82, LNCS}, 144 (1982),  24--31.
%
\bibitem{Co}
CoCoA Team,
{CoCoA: a system for doing computations in Commutative Algebra.}
Available at \verb|http://cocoa.dima.unige.it/|
%
\bibitem{CGKW00}
R.~M.~Corless, M.~W.~Giesbrecht, I.~S.~Kotsireas,  and S.~M.~Watt,
\textit{Numerical implicitization of parametric hypersurfaces with linear algebra.}
Proc. Artificial Intelligence with Symbolic Computation, (AISC 2000),
Lecture Notes in Artificial Intelligence, 1930 (2000), Springer Verlag, 174--183.
%
\bibitem{Datta10}
B.~N. Datta,
\textit{Numerical Linear Algebra and Applications}. SIAM, Philadelphia, 2010.
%
\bibitem{DH93}
J.~W. Demmel and  N.~J. Higham,
\textit{Improved Error Bounds for Underdetermined System Solvers.}
 SIAM J. Matrix Anal. Appl. 14 (1993), no. 1, 1--14.
%
\bibitem{Fa10}
C.~Fassino,
\textit{Vanishing Ideal of Limited Precision Points.}
J. Symbolic Comput. 45 (2010), no. 1, 19--37.
%
\bibitem{GVL96}
G.~H. Golub and  C.~F. Van Loan,
\textit{Matrix Computations. Third edition.} Johns Hopkins Studies in the Mathematical Sciences. Johns Hopkins University Press, Baltimore, MD, 1996.
%
\bibitem{GSL}
GSL(release 1.15)- GNU Scientific Library.
Available at \verb|http://www.gnu.org/software/gsl/|
%
\bibitem{HKPP09}
 D. Heldt, M. Kreuzer, S. Pokutta, and  H. Poulisse,
\textit{Approximate computation of zero-dimensional polynomial ideals.}
J. Symbolic Comput. 44 (2009), no. 11, 1566--1591.
%
\bibitem{HP92}
Y.~P. Hong and C.~T. Pan,
\textit{Rank-Revealing QR Factorizations and the Singular Value Decomposition.}
Math. Comp. 58 (1992), no. 197, 213--232.
%
\bibitem{KP11}
M. Kreuzer and  H. Poulisse,
\textit{Subideal border bases.} 
Math. Comp. 80 (2011), no. 274, 1135--1154.
%
\bibitem{KR00}
M. Kreuzer, L. Robbiano,
\textit{Computational Commutative Algebra 1}. Springer-Verlag, Berlin, 2000.
%
\bibitem{KR05}
M. Kreuzer and L.  Robbiano,
\textit{Computational Commutative Algebra 2}. Springer-Verlag, Berlin, 2005.
%
\bibitem{M99}
B. Mourrain,
\textit{A New Criterion for Normal Form Algorithms.}
Proc. AAECC, LNCS 1719 (1999), 430--443.
%
\bibitem{RA}
L.~Robbiano and J.~Abbott (Eds.),
\textit{Approximate Commutative Algebra.}
Text \& Monographs in Symbolic Computation, Springer, 2010.
%
\bibitem{S07}
T.~Sauer, 
\textit{Approximate varieties, approximate ideals and dimension reduction.}
Numer. Algor. 45 (2007), 295--313.
%
\bibitem{St}
 H.~J. Stetter,
\textit{Numerical Polynomial Algebra}. SIAM, Philadelphia, PA, USA, 2004.
%
\bibitem{St06}
H.~J. Stetter, 
\textit{``Approximate Commutative Algebra" - an ill-chosen name for an important discipline.}
ACM Commun. Comput. Algebra 40 (2006), no. 3.
%
\bibitem{T}
M. Torrente,
\textit{Application of algebra in the oil industry.}
{ Ph.D. Thesis}, Scuola Normale Superiore, Pisa, 2009.
%
\bibitem{WW}
H.~F. Walker and  L.~T. Watson,
\textit{Least-change secant update methods for underdetermined systems.}
SIAM J. Numer. Anal. 27 (1990), no. 5, 1227--1262. 
%
\end{thebibliography}


\end{document}